\documentclass[12pt,reqno]{amsart}

\usepackage{times}
\usepackage[T1]{fontenc}
\usepackage{mathrsfs}
\usepackage{latexsym}
\usepackage[dvips]{graphics}
\usepackage{epsfig}
\usepackage{amsmath,amsfonts,amsthm,amssymb,amscd}
\input amssym.def
\input amssym.tex
\usepackage{color}
\usepackage{hyperref}
\usepackage{color}
\usepackage{breakurl}
\usepackage{comment}
\newcommand{\bburl}[1]{\textcolor{blue}{\url{#1}}}

\newtheorem{theorem}{Theorem}[section]

\newcommand\be{\begin{equation}}
\newcommand\ee{\end{equation}}
\newcommand\bea{\begin{eqnarray}}
\newcommand\eea{\end{eqnarray}}
\newcommand\bi{\begin{itemize}}
\newcommand\ei{\end{itemize}}
\newcommand\ben{\begin{enumerate}}
\newcommand\een{\end{enumerate}}
\newcommand\bc{\begin{center}}
\newcommand\ec{\end{center}}
\newcommand\ba{\begin{array}}
\newcommand\ea{\end{array}}

\newtheorem*{cond1}{Condition 1}
\newtheorem*{cond2}{Condition 2}

\newtheorem{thm}{Theorem}[section]
\newtheorem{conj}[thm]{Conjecture}

\newtheorem{lem}[thm]{Lemma}

\newtheorem{defi}[thm]{Definition}

\theoremstyle{definition}
\newtheorem{rek}[thm]{Remark}

\numberwithin{equation}{section}

\begin{document}

\title{The Generalized Zeckendorf Game}

\author{Paul Baird-Smith}
\email{\textcolor{blue}{\href{mailto:paul.bairdsmith@gmail.com}{paul.bairdsmith@gmail.com}}}
\address{Department of Computer Science, University of Texas at Austin, Austin, TX}

\author{Alyssa Epstein}
\email{\textcolor{blue}{\href{mailto:alye@stanford.edu}{alye@stanford.edu}}}
\address{Department of Law, Stanford University, Stanford, CA}

\author{Kristen Flint}
\email{\textcolor{blue}{\href{mailto:kflint1101@gmail.com}{kflint1101@gmail.com}}}
\address{Department of Mathematics, Carnegie Mellon University, Pittsburgh, PA 15213}

\author{Steven J. Miller}
\email{\textcolor{blue}{\href{mailto:sjm1@williams.edu}{sjm1@williams.edu}},  \textcolor{blue}{\href{Steven.Miller.MC.96@aya.yale.edu}{Steven.Miller.MC.96@aya.yale.edu}}}
\address{Department of Mathematics and Statistics, Williams College, Williamstown, MA 01267}



\date{\today}

\thanks{The authors were partially supported by NSF grants DMS1265673 and DMS1561945, the Claire Booth Luce Foundation, and Carnegie Mellon University. We thank the students from the Math 21-499 Spring '16 research class at Carnegie Mellon and the participants from CANT 2016 and the 18\textsuperscript{th} Fibonacci Conference, especially Russell Hendel, for many helpful conversations.}

\begin{abstract} Zeckendorf proved that every positive integer $n$ can be written uniquely as the sum of non-adjacent Fibonacci numbers; a similar result, though with a different notion of a legal decomposition, holds for many other sequences. We use these decompositions to construct a two-player game, which can be completely analyzed for linear recurrence relations of the form $G_n = \sum_{i=1}^{k} c G_{n-i}$ for a fixed positive integer $c$ ($c=k-1=1$ gives the Fibonaccis). Given a fixed integer $n$ and an initial decomposition of $n = n G_1$, the two players alternate by using moves related to the recurrence relation, and whomever moves last wins. The game always terminates in the Zeckendorf decomposition, though depending on the choice of moves the length of the game and the winner can vary. We find upper and lower bounds on the number of moves possible; for the Fibonacci game the upper bound is on the order of $n\log n$, and for other games we obtain a bound growing linearly with $n$. For the Fibonacci game, Player 2 has the winning strategy for all $n > 2$. If Player 2 makes a mistake on his first move, however, Player 1 has the winning strategy instead. Interestingly, the proof of both of these claims is non-constructive.
\end{abstract}

\maketitle

\tableofcontents


\section{Introduction}\label{section}


\subsection{History}

Familiar from many varied contexts, from mathematical biology to Pascal's triangle, the Fibonacci numbers are an incredibly fascinating and famous sequence. Looking at this sequence, Zeckendorf \cite{Zeckendorf} generated a beautiful theorem: each positive integer $n$ can be written uniquely as the sum of distinct, non-adjacent Fibonacci numbers. This is called the \textit{Zeckendorf decomposition} of $n$ and requires the Fibonacci numbers to be defined as $F_1 = 1, F_2 = 2, F_3 = 3, F_4 = 5, \dots$ instead of the usual $1, 1, 2, 3, 5, \dots$ for uniqueness. The Zeckendorf theorem has been generalized many times (see for example \cite{Br,Ca,CFHMN2,CFHMNPX,Day,Dem,Dor,FGNPT,Fr,GTNP,Ha,Ho,Ke,LT,Len,Ste1,Ste2}); we follow the terminology used by Miller and Wang \cite{MW1}.

\begin{defi}\label{defn:goodrecurrencereldef}\label{def:goodrecurrence} We say a sequence $\{H_n\}_{n=1}^\infty$ of positive integers is a \textbf{Positive Linear Recurrence Sequence (PLRS)} if the following properties hold.

\ben
\item \emph{Recurrence relation:} There are non-negative integers $L, c_1, \dots, c_L$\label{c_i} such that $$H_{n+1} \ = \ c_1 H_n + \cdots + c_L H_{n+1-L},$$ with $L, c_1$ and $c_L$ positive.
\item \emph{Initial conditions:} $H_1 = 1$, and for $1 \le n < L$ we have
$$H_{n+1} \ =\
c_1 H_n + c_2 H_{n-1} + \cdots + c_n H_{1}+1.$$
\een

We call a decomposition $\sum_{i=1}^{m} {a_i H_{m+1-i}}$\label{a_i} of a positive integer $N$ (and the sequence $\{a_i\}_{i=1}^{m}$) \textbf{legal}\label{legal} if $a_1>0$, the other $a_i \ge 0$, and one of the following two conditions holds.

\begin{cond1}\label{legalcond1}
We have $m<L$ and $a_i=c_i$ for $1\le i\le m$.
\end{cond1}

\begin{cond2}
There exists $s\in\{1,\dots, L\}$ such that
\begin{equation}\label{eq:legalcondition2}
a_1\ = \ c_1,\ a_2\ = \ c_2,\ \cdots,\ a_{s-1}\ = \ c_{s-1}\ {\rm{and}}\ a_s<c_s,
\end{equation}
$a_{s+1}, \dots, a_{s+\ell} \ = \  0$ for some $\ell \ge 0$,
and $\{b_i\}_{i=1}^{m-s-\ell}$ (with $b_i = a_{s+\ell+i}$) is legal.
\end{cond2}

If\space $\sum_{i=1}^{m} {a_i H_{m+1-i}}$ is a legal decomposition of $N$, we define the \textbf{number of summands}\label{summands} (of this decomposition of $N$) to be $a_1 + \cdots + a_m$.
\end{defi}

Informally, a legal decomposition is one where we cannot use the recurrence relation to replace a linear combination of summands with another summand, and the coefficient of each summand is appropriately bounded; other authors \cite{DG,Ste1} use the phrase $G$-ary decomposition for a legal decomposition, and sum-of-digits or summatory function for the number of summands. For example, if $H_{n+1} = 2 H_n + 3 H_{n-1} + H_{n-2}$, then $H_5 + 2 H_4 + 3 H_3 + H_1$ is legal, while $H_5 + 2 H_4 + 3 H_3 + H_2$ is not (we can replace $2 H_4 + 3 H_3 + H_2$ with $H_5$), nor is $7H_5 + 2H_2$ (as the coefficient of $H_5$ is too large). Note the Fibonacci numbers are just the special case of $L=2$ and $c_1 = c_2 = 1$.

\begin{theorem}[Generalized Zeckendorf's Theorem for PLRS]\label{thm:genZeckendorf} Let $\{H_n\}_{n=1}^\infty$ be a \emph{Positive Linear Recurrence Sequence}. Then there is a unique legal decomposition for each positive integer $N\ge 0$.
\end{theorem}

For more on generalized Zeckendorf decompositions, see the references mentioned earlier, and for proofs of Theorem \ref{thm:genZeckendorf}, see \cite{GT}. This paper aims to introduce a game on these generalized decompositions and prove a variety of properties of such games.


\subsection{New Work}

We first introduce some notation. When we write $\{1^n\}$ or $\{{F_1}^n\}$, we mean n copies of $1$, the first Fibonacci number. If we have 3 copies of $F_1$, 5 copies of $F_3$, and 9 copies of $F_5$, we write either $\{{F_1}^3 \wedge {F_3}^5 \wedge {F_5}^9 \}$ or $\{1^3 \wedge 3^5 \wedge 8^9\}$. We use similar notation for games arising from other recurrences. We start with the Fibonacci case, and then generalize.


\begin{defi}[The Two Player Zeckendorf Game]\label{defi:zg}
At the beginning of the game, there is an unordered list of $n$ 1's. Let $F_1 = 1, F_2 = 2$, and $F_{i+1} = F_i + F_{i-1}$; therefore the initial list is $\{{F_1}^n\}$. On each turn, a player can do one of the following moves.
\begin{enumerate}
\item If the list contains two consecutive Fibonacci numbers, $F_{i-1}, F_i$, then a player can change these to $F_{i+1}$. We denote this move $\{F_{i-1} \wedge F_i \rightarrow F_{i+1}\}$.
\item If the list has two of the same Fibonacci number, $F_i, F_i$, then
\begin{enumerate}
\item if $i=1$, a player can change $F_1, F_1$ to $F_2$, denoted by $\{F_1 \wedge F_1 \rightarrow F_2\}$,
\item if $i=2$, a player can change $F_2, F_2$ to $F_1, F_3$, denoted by $\{F_2 \wedge F_2 \rightarrow F_1 \wedge F_3\}$, and
\item if $i \geq 3$, a player can change $F_i, F_i$ to $F_{i-2}, F_{i+1}$, denoted by $\{F_i \wedge F_i \rightarrow F_{i-2}\wedge F_{i+1} \}$.
\end{enumerate}
\end{enumerate}
The players alternative moving. The game ends when no moves remain.
\end{defi}

The moves of the game are derived from the recurrence, either combining terms to make the next in the sequence or splitting terms with multiple copies. A proof that this game is well defined, ends at the Zeckendorf decomposition, has a sharp lower bound on the number of moves of $n-Z(n)$, and has an upper bound on the order of $n \log n$ can be found in \cite{BEFMcant}. The same paper also proves the following theorem, the proof of which we reproduce.

\begin{thm}\label{thm:playertwowins}
For all $n>2$, Player 2 has the winning strategy for the Zeckendorf Game.\footnote{If $n=2$, there is only one move, and then the game is over.}
\end{thm}

Interestingly, our proof is non-constructive; we show that Player 2 has a winning strategy but we cannot find it.\footnote{In principle one could enumerate all games for a given starting $n$, but at present we can only analyze a fixed $n$ by brute force.} We can however expand on this result, giving a lemma proved in a similar manner; again, the strategy is non-constructive.

\begin{lem}\label{lem:p1retaliate}
For all $n>3$, Player 1 has the winning strategy if Player 2 makes the wrong move on his first turn. \footnote{If $n=2$, Player 2 never moves. If $n=3$, Player 2 cannot make a mistake as there is only one available move. We thank Russell Hendel for asking the question on how early in the game we can have Player 2 make a bad move.}
\end{lem}

The Zeckendorf game as described so far only concerns a game on the Fibonacci sequence. However, using Theorem \ref{thm:genZeckendorf}, we can create new games on other positive linear recurrence sequences, though at present we can only obtain results similar to the Fibonacci case for special recurrences. We define a few terms before proposing a Generalized Zeckendorf game.

\begin{defi}[$k$-nacci Numbers]\label{defi:knaccinums}
We call any sequence defined by a recurrence $S_{i+1} = S_i +S_{i-1}+ \cdots + S_{i-k}$ a $k$-nacci sequence. The initial conditions are as follows: $S_1 = 1$, and for $1 \le n < k+1$ we have $S_{i+1} = S_i + S_{i-1} + \cdots + S_1 + 1.$ The terms $S_t$ are called $k$-nacci numbers.
\end{defi}

The $k$-nacci sequence generalizes the Fibonacci sequence by extending the number of prior consecutive terms added up to get the next in the sequence. The Fibonaccis may be viewed as $1$-naccis, and the Tribonaccis may be viewed as $2$-naccis.

\begin{defi}[Generalized $k$-nacci Numbers]\label{defi:genknaccis}
We call any sequence defined by a recurrence $S_{i+1} = cS_i +cS_{i-1}+ \cdots + cS_{i-k}$ a generalized $k$-nacci sequence with constant $c$. The initial conditions are as follows: $S_1 = 1$, and for $1 \le n < k+1$ we have $S_{i+1} = cS_i + cS_{i-1} + \cdots + cS_1 + 1$. The terms $S_t$ are called $(c,k)$-nacci numbers.
\end{defi}

Generalized $k$-nacci numbers apply a constant $c$ in front of each term added to create the next in the sequence. As an example, the Fibonaccis are $(1,1)$-nacci numbers.

\begin{defi}[The Two-Player Generalized Zeckendorf Game]\label{defi:genzeckgame}
Two people play the Generalized Zeckendorf game for the $k$-nacci numbers. At the beginning of the game, we have an unordered list of $n$ 1's. If $i < k+1$, $S_{i+1} = cS_i + cS_{i-1}+\cdots+cS_1 + 1$. If $i\geq k$, $S_{i+1} = cS_i +cS_{i-1}+\cdots+ cS_{i-k}$. Therefore our initial list is $\{S_1^n\}$. On each turn we can do one of the following moves.
\begin{enumerate}
\item If our list contains $k+1$ consecutive k-nacci numbers each with multiplicity $c$, then we can change these to $S_{i+1}$. We denote this move $\{cS_{i-k}\wedge cS_{i-k+1}\wedge \cdots \wedge cS_i \rightarrow S_{i+1}\}$.
\item If our list contains consecutive $k$-nacci numbers with multiplicity $c$ up to an index less than or equal to $k$, and $S_1$ with multiplicity $c+1$, we can do the move $\{(c+1)S_{1}\wedge cS_2 \wedge \cdots \wedge cS_{i} \rightarrow S_{i+1}\}$.
\item If the list has $c+1$ of the same $k$-nacci number $S_i$, then
\begin{enumerate}
\item if $i=1$, then we can change $(c+1)S_1$ to $S_2$, denoting this move $\{(c+1)S_1 \rightarrow S_2\}$;
\item if $1<i< k+1$, then we can change $(c+1)S_i$ to $S_{i+1}$, denoted by $\{S_i \wedge S_i \rightarrow S_{i+1}\}$;
\item if $i= k+1$, then we can do the move $\{(c+1)S_i \rightarrow S_{i+1} \wedge S_1\}$; and
\item if $i > k+1$, then we can do the move $\{(c+1)S_i \rightarrow S_{i+1} \wedge cS_{i-k-1} \}$.
\end{enumerate}
\end{enumerate}
Players alternate moving until no moves remain.
\end{defi}

Again, we may wonder whether this game is well defined and ends at the Generalized Zeckendorf decomposition for the given recurrence. It is, as we prove through the next theorem.

\begin{thm}[The Generalized Zeckendorf Game is Well-Defined]\label{thm:genzwd}
Every Generalized Zeckendorf game terminates within a finite number of moves at the Generalized Zeckendorf decomposition.
\end{thm}

The proof of this theorem proceeds by defining a monovariant that enables another useful result about the length of games.

\begin{lem}[Upper Bound on the Generalized Zeckendorf Game]\label{lem:upboundgenzeck}
All Generalized Zeckendorf Games (apart from Fibonacci) end in at most $2n - GZD(n) - IGZD(n)$ moves, where $GZD(n)$ is the number of terms in the Generalized Zeckendorf Decomposition of $n$ and $IGZD(n)$ is the sum of the indices in the same decomposition.
\end{lem}

It is worth noting that the monovariant for the Generalized Zeckendorf Game gives a better upper bound than in the Fibonacci case. This monovariant does not apply only over the Fibonacci sequence. Additionally, if we try to expand the scope of the Generalized Zeckendorf game to PLRS other than generalized $k$-nacci numbers, we either struggle with to find \textit{any} monovariant, not just a nice one, or we cannot define a splitting move with the recurrence. Future work can try to address these problems.


\section{The Zeckendorf Game}

As someone must always make the final move, and as the game always ends at the Zeckendorf decomposition, there are no ties. Therefore one player or the other has a winning strategy for each $n$. This section is devoted to proofs of winning strategies. Specifically, Player 2 has the winning strategy for all $n >2$, the statement of Theorem \ref{thm:playertwowins}. If Player 2 makes an error on his first move, Player 1 can force a victory. For the proof of the both claims, we use a visual aid provided in Figure \ref{fig:tree}.

\begin{figure}[h]
\includegraphics[scale=.8]{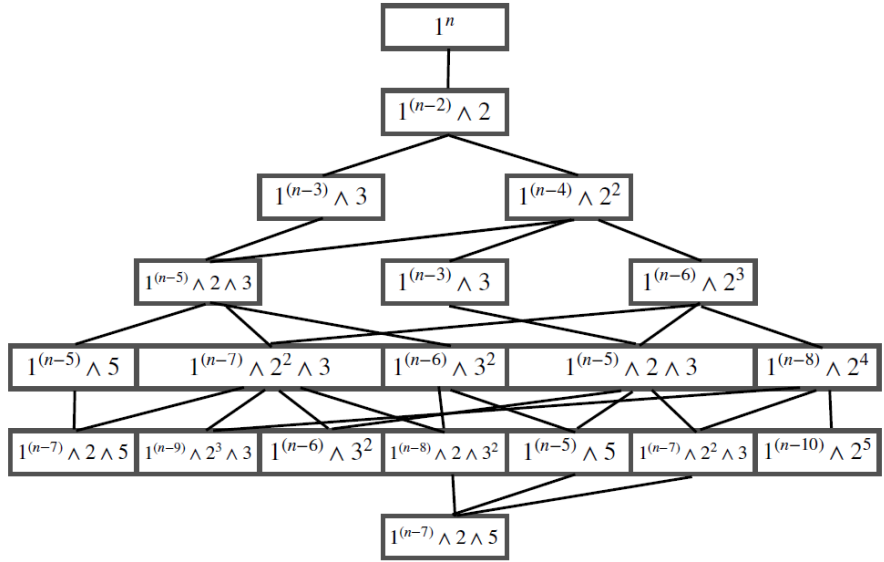}
\caption{Tree depicting the general structure of the first several moves of the Zeckendorf game.}
\label{fig:tree}
\end{figure}

\begin{proof}
For the entirety of this proof, we color boxes in Figures \ref{fig:treeproof1} and \ref{fig:treeproof2} red if we are assuming or have deduced that Player 1 has a winning strategy from the node and blue if Player 2 does.

Assume that Player 1 wins the game, in other words, that Player 1 has the winning strategy from the initial node at the top of the tree. From this game state, only one move can be made, regardless of the size of $n$. For Player 1 to have the winning strategy for the whole game, it thus follows that they must have one from the only node on the second row of the tree. Player 2 moves next, so Player 1 must have the winning strategy from all nodes in row 3; if not, Player 2 would simply move to the one where Player 1 did not have the winning strategy. The first node on the third row has only one child, so its descendant in row 4 must also have a winning strategy for Player 1. Player 2 moves on the nodes in row 4, so Player 1 must have a winning strategy from all three children of $\{1^{(n-5)}\wedge 2\wedge 3\}$ in row 5.

	One of the children in row 5 is $\{1^{(n-5)}\wedge 5\}$. Observe that in row 6 of the tree this same game state may be found. If Player 1 has the winning strategy from that state in row 5, by following the same strategy, Player 2 can arrive at victory from the node in row 6 by reasons of parity. Player 2 also therefore must possess the winning strategy from the only child of that node in row 7. Now, this implies that any parent of that game state in row 6 must also bear a winning strategy for Player 2 because, as it is their turn in row 6, they could move to $\{1^{(n-7)}\wedge 2 \wedge 5 \}$ in row 7 from each of those. Accordingly, Player 2 must be able to win from $\{1^{(n-8)}\wedge 2 \wedge 3^{(2)} \}$ in row 6. Yet, we have now found that both children in row 6 of $\{1^{(n-6)}\wedge 3^{(2)}\}$ in row 5 have winning strategies for Player 2, contradicting our early claim that the node held a winning strategy for Player 1. The theorem is thus proven for all $n$ whose game tree possesses 7 layers or more ($n\geq 9$). For the small cases of $2<n<9$, computer code such as the one referenced in Appendix \ref{sec:Jcode} can show that Player 2 has the winning strategy by brute force.
\end{proof}

\begin{figure}[h!]
\includegraphics[scale=.8]{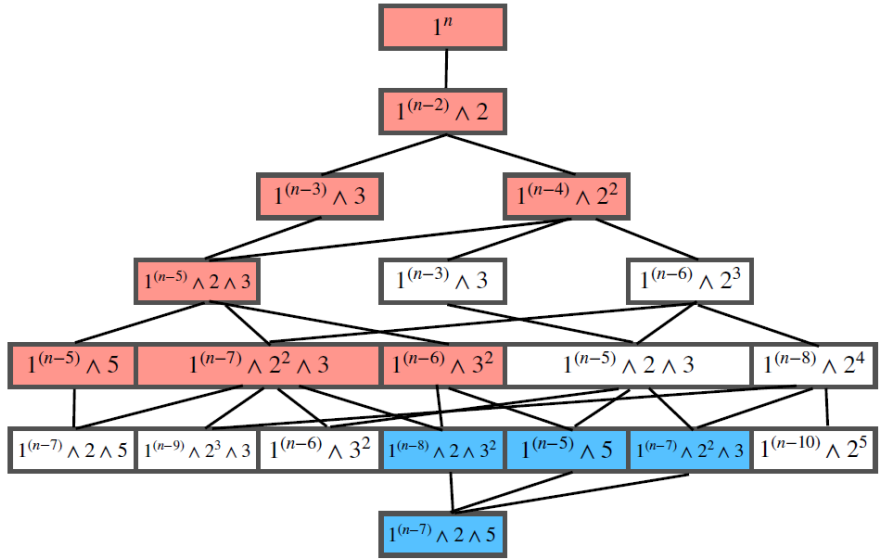}
\caption{Tree depicting the proof of Theorem \ref{thm:playertwowins}. Red boxes have a winning strategy for Player 1, and blue boxes indicate a winning strategy for Player 2.}
\label{fig:treeproof1}
\end{figure}

We follow a similar proof strategy for Lemma \ref{lem:p1retaliate}.

\begin{proof}
Suppose that Player 2 has a winning strategy from the second box in the third row of the game tree on $n$. Then, since Player 1 makes the move from that node, all of the descendants must have a winning strategy for Player 2. Then, the middle node, having only one child, must have that child also be a winning strategy for Player 2. Notice though that $\{1^{(n-5)}\wedge 2 \wedge 3\}$ can be found on both rows 4 and 3. Since Player 2 has the winning strategy from that node in row 4, it follows that Player 1 must have a strategy from row 3. This is a contradiction, and it shows that Player 2 can compromise their potential victory as early as their first move (the second move of the game). Again, this proof works for $n$ sufficiently large ($n \geq 5$). For the special case of $n=4$, Player 1 wins immediately after Player 2 executes the wrong move.
\end{proof}

\begin{figure}[h!]
\includegraphics[scale=.8]{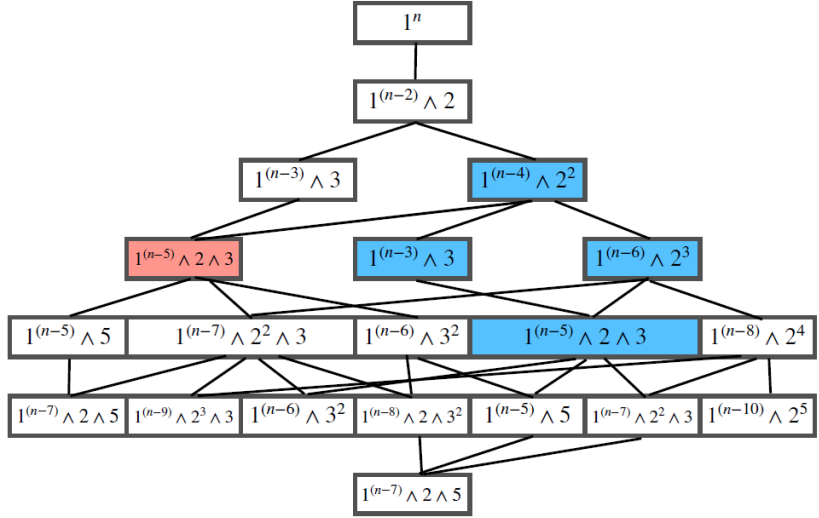}
\caption{Tree depicting the proof of Lemma \ref{lem:p1retaliate}. Red boxes have a winning strategy for Player 1, and blue boxes indicate a winning strategy for Player 2.}
\label{fig:treeproof2}
\end{figure}

These results from Theorem \ref{thm:playertwowins} and Lemma \ref{lem:p1retaliate} are both interesting and surprising. Game trees for large $n$ have many nodes, with no obvious path to victory for either player (see Figure \ref{fig:smalltree} for $n=9$ and Figure \ref{fig:bigtree} for $n=14$ for an example of how quickly the number of nodes grows). These are also both only existence proofs, without indication of how Player 2 or Player 1 should move in general (except that Player 2 should move in a specific way on their first turn). The uncertainty in the achievement of these winning strategies makes the game seem less unfair in play. In fact, random simulations show that Players 1 and 2 win about as often as each other.

\begin{figure}[h!!]
\includegraphics[scale=.85]{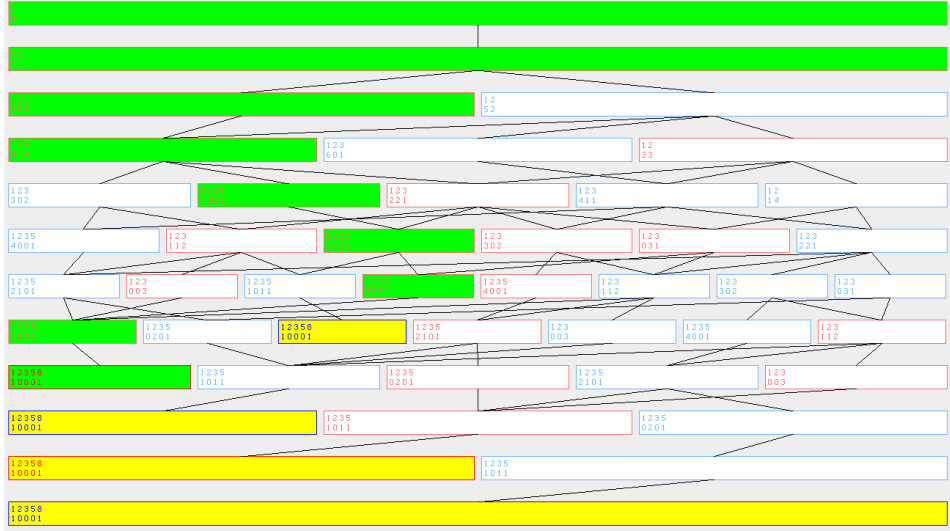}
\caption{Game tree for $n=9$, showing a winning path in green. Image courtesy of the code referenced in Appendix \ref{sec:Jcode}.}
\label{fig:smalltree}
\end{figure}

\begin{figure}[h!!]
\includegraphics[scale=.85]{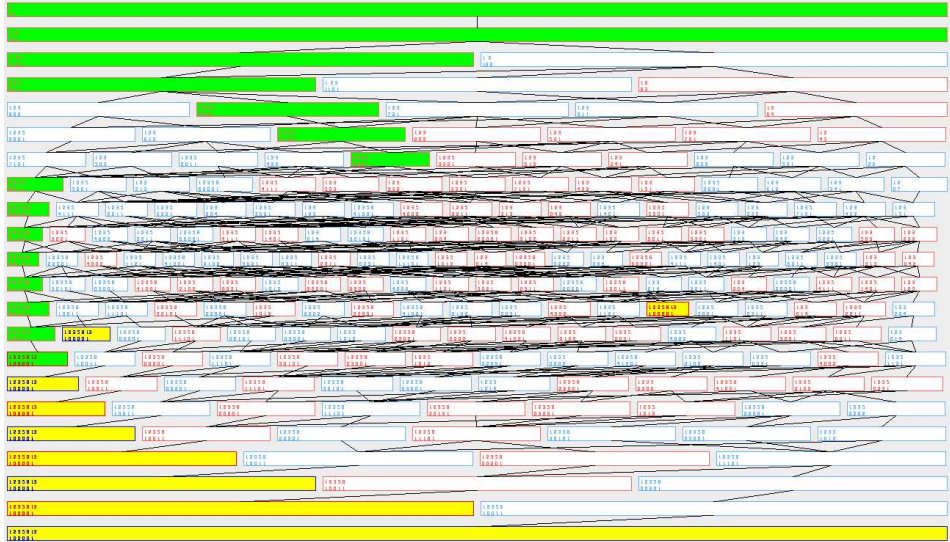}
\caption{Game tree for $n=14$, showing a winning path in green. Image courtesy of the code in Appendix \ref{sec:Jcode}.}
\label{fig:bigtree}
\end{figure}

\section{The Generalized Zeckendorf Game}


\subsection{The Generalized Game is Playable}
This section examines the generalization of the Zeckendorf Game to a particular class of positive linear recurrence relations called the generalized $k$-naccis (see Definition \ref{defi:genknaccis}). The generalized game's rules are set out in Definition \ref{defi:genzeckgame}. Some of the rules, particularly the ones on the splitting moves, seem un-intuitive. We assure the reader that these moves can be derived from the recurrence without too much difficulty. We start with a lemma that helps prove that this game is well defined, the statement of Theorem \ref{thm:genzwd}.

\begin{lem}[Generalized Zeckendorf Monovariant]\label{lem:gzmono}
The sum of the number of terms plus the sum of the indices of those terms is a monovariant for the Generalized Zeckendorf Game, in all cases but the game on the Fibonacci relation.
\end{lem}

\begin{proof}
We define a monovariant $\delta$ on this game, where $\delta$ is the sum of the number of terms and the indices of the set in any given turn. We prove that our moves always decrease this monovariant except in the Fibonacci case ($c=1, k=1$). We note that $\delta$ is an additive function, so we can just examine what $\delta$ does to the subset of terms affected by the moves. We note that we follow the move numbering established in Definition \ref{defi:genzeckgame}.

Before we do move (1), we have a value in the summands we are using of $\delta(\{cS_{i-k}\wedge cS_{i-k+1}\wedge\cdots \wedge cS_i\}) = (k+1)c + c(i-k) +\cdots + ci \geq c + 2ci$. After, we have a value of $\delta(\{S_{i+1}\}) = 1 + i + 1 = i+2$. We note that this move only happens when $i > k \geq 1$, so $c+2ci > 1+ 2i > i+2$ for all positive $c$ and all valid $i$. Therefore the monovariant decreases when move (1) is executed.

Before we do move (2), we have a value of $\delta(\{(c+1)S_{1}\wedge cS_2 \wedge \cdots \wedge cS_{i}\}) = ci + 1 + (c+1) + 2c + \cdots + ic \geq ci + c + 2$. After the move, we have a value of $\delta(\{S_{i+1}\}) = 1 + i+1 = i+2$. For all positive $c$ and for all $i$, $ci + c +2 > i +2$. Therefore the monovariant decreases when move (2) is executed.

Before we do move (3a), we have a value of $\delta(\{(c+1)S_1\}) = c+1 +c +1 = 2c +2$. After, we have $\delta(\{S_2\}) = 1 + 2 = 3$. For all positive $c$, $2c + 2 > 3$, so the monovariant decreases.

Before we do moves (3b), (3c), and (3d), we have a value of $\delta(\{(c+1)S_i\}) = c+1 + (c+1)i = c + 1 + ci + i$. After doing (3b), we get $\delta(\{S_{i+1}\}) = 1 + i + 1 = i + 2$. We see that $c + 1 + ci + i \geq 2i + 2 > i +2$ for all positive $c$, so the monovariant holds. After doing (3c), we get $\delta(\{S_1 + S_{i+1}\}) = 2 + 1 + i + 1 = i + 4$. If $c=1, i =2$ (the Fibonacci case), then $c + 1 + ci + i = 6 = i + 4$. However, if we assume that $i=k+1 > 2$, then $c+ 1 + ci + i > 2i + 2 > i + 4$. So if $k>1$, then the monovariant holds for any positive $c$. On the other hand, if we require that $c>1$, then $c + 1 + ci + i \geq 3 + 3i>i + 4$ for all $i$ (and hence $k$), so the monovariant holds for all (3b) except when the recurrence relation is Fibonacci. After (3d) we have $\delta(\{S_{i+1}\wedge cS_{i-k-1}\}) = c+1 + i + 1 + ci -ck -c = ci + 2 + i -ck < ci + 2 + i$. We know that $c+1 + ci + i \geq 2 + ci + i$ for all $c$, so the monovariant holds in this case as well.

Since the value of delta on the summands employed in the moves always decreases, this is truly a monovariant (in all cases but Fibonacci).
\end{proof}

We can now prove Theorem \ref{thm:genzwd}.

\begin{proof}
Given the monovariant $\delta$ established in Lemma \ref{lem:gzmono} and the monovariant developed for the special case of the Fibonacci recurrence shown in \cite{BEFMcant}, we know that there are no repeat turns in the Generalized Zeckendorf Game. Moreover, since there are only a finite number of partitions of $n$ among any positive linear recurrence sequence bounded by $n$, this means that the game must end somewhere. The game must end at the Generalized Zeckendorf decomposition laid out in Theorem \ref{thm:genZeckendorf} because if the recurrence relation can be applied again, the game has not terminated, and if there are more than $c$ duplicates of any term, the game has not terminated. If the recurrence relation cannot be applied, and there are at most $c$ of any term, this is exactly the Generalized Zeckendorf decomposition by its uniqueness. Therefore the game is well-defined.
\end{proof}

\begin{rek}
Having different constants $c_i$ in Generalized Zeckendorf games for generalized $k$-naccis is beyond the scope of this paper because of the complexity of the structure of the Generalized Zeckendorf decomposition for these recurrence relations. Also, other relations do not necessarily define splitting moves, so the game would be deterministic (and therefore boring).
\end{rek}


\subsection{Bounds on the Length of Generalized Zeckendorf Games}

Like in the Fibonacci Zeckendorf Game, we consider upper and lower bounds for the Generalized Zeckendorf Game. We begin with the proof of Lemma \ref{lem:upboundgenzeck}.

\begin{proof}
This follows immediately by the existence of the monovariant and the fact that the monovariant decreases by at least 1 each time.
\end{proof}

\begin{defi}[Tribonacci Sequence]
We define the Tribonacci Sequence as the recurrence relation $R_{n+1} = R_n+ R_{n-1} +R_{n-2}$ with the base cases $R_1 = 1, R_2=2, R_3=7$.
\end{defi}

\begin{lem}[A Deterministic Tribonacci Game]\label{lem: detTrib}
We play a Tribonacci Generalized Zeckendorf Game where we always act on the greatest valued summand with an available move. This game is deterministic and will not involve any splitting moves.
\end{lem}

\begin{proof}
Suppose our largest integer is $R_n$. If $n=1$, then if there's an available move using $R_n$, it must be $\{R_1\wedge R_1 \rightarrow R_2\}$. If $n=2$, if there's a valid move on $R_n$, it must be $\{R_2\wedge R_2 \rightarrow R_3\}$. For both the case when $n=1, n=2$, we cannot have consecutive moves or else we would have to have $R_3$, a contradiction to $R_n$ being the largest integer. If $n=3$, at first glance it appears that there are two options for moves: $\{R_1\wedge R_2 \wedge R_3 \rightarrow R_4\}$ and $\{R_3\wedge R_3 \rightarrow R_1 \wedge R_4\}$. However, to get $\{R_3, R_3\}$, the turn before we must have had either $\{R_3\wedge R_2\wedge R_2\}$ or $\{R_3\wedge R_6 \wedge R_6\}$. If we had the latter set, then now we must have $R_7$ as a summand, contradicting the claim that $n=3$. If the turn before we had $\{R_3\wedge R_2\wedge R_2\}$ and we had an $\{R_1\}$, then we never would have joined the two $R_2$s because a consecutive move was available. If we did not have an $\{R_1\}$ then in order to create $\{R_3\wedge R_2 \wedge R_2\}$, the turn prior must have either been $\{R_3\wedge R_5 \wedge R_5 \}$, a contradiction, or $\{R_3 \wedge R_2 \wedge R_1 \wedge R_1\}$. So we would not have gotten $\{R_3\wedge R_2\wedge R_2\}$ if we had acted on the largest integer because we would have added consecutives the turn before. So really our only move option on the largest integer $R_3$ is $\{R_1\wedge R_2 \wedge R_3 \rightarrow R_4\}$. If $n>3$, then we might think we could do either $\{R_{n-2}\wedge R_{n-1} \wedge R_n \rightarrow R_{n+1}\}$ or $\{R_n \wedge R_n \rightarrow R_{n+1} \wedge R_{n-3}\}$. However, to arrive at $\{R_n \wedge R_n\}$, a turn before we must have had $\{R_n \wedge R_{n-1}\wedge R_{n-1}\}$ and no $\{R_{n-2}\}$ (or else we could have done a consecutive move on $\{R_n\}$ last turn), $\{R_n \wedge R_{n-1} \wedge R_{n-2} \wedge R_{n-3}\}$, which means we should already have added consecutives, or $\{R_n \wedge R_{n+3}\wedge R_{n+3}\}$, which is automatically a contradiction. But to get to $\{R_n \wedge R_{n-1}\wedge R_{n-1}\}$, we must have had either $\{R_n \wedge R_{n-1} \wedge R_{n+2} \wedge R_{n+2}\}$ a turn before, a contradiction, or $\{R_n \wedge R_{n-1} \wedge R_{n-2} \wedge R_{n-2}\}$ or $\{R_{n} \wedge R_{n-1} \wedge R_{n-2} \wedge R_{n-3} \wedge R_{n-4}\}$. In either of the latter cases, we would have done a consecutive move and there was no way for us to get $\{R_n \wedge R_n \}$ where $R_n$ is the largest integer and we always did a move on the largest integer with an available move.

Therefore we have shown that the game is deterministic because regardless of what the largest integer is, we only have one valid move. Moreover we have shown that all of the moves we will do in the deterministic game described either join 2 base case elements into another base element or add consecutives. We will never have a splitting move (or duplicates of $R_n$ for $n >3$).
\end{proof}

\begin{conj}[Deterministic Game is Best for Tribonacci]
The exact minimum amount of moves in the Tribonacci game is achieved by the greedy algorithm described in the deterministic game in Lemma \ref{lem: detTrib}.
\end{conj}



\begin{lem}[Lower Bounds for the Tribonacci Game]\label{lem:lowboundtrib}
All games end in at least $(n-GZD(n))/2$ moves.
\end{lem}

\begin{proof}
The most we can decrease terms by in any given move is 2 (if we combine consecutives). If at every step we used this move, we would arrive at the Generalized Zeckendorf decomposition of $n$ in $(n-GZD(n))/2$ moves.
\end{proof}

\begin{rek}[Difficulties in Lower Bounds]\label{rek:hardbounds}
In the Fibonacci Case, a lower bound on the number of moves was easy to figure out because all of the moves either changed the number of terms by $1$ or by $0$. We were also able to show that games exist that always decrease the number of terms by nonzero amounts. Though it is fairly certain that we could always find a game that decreases the number of terms by nonzero amounts in the Generalized Zeckendorf Game by using a greedy algorithm such as is proved for the special case of the Tribonacci Game in Lemma \ref{lem: detTrib}, the other moves vary from $1$ to $ck$ in the number of terms they remove. This makes it difficult to get a sharp lower bound without knowing the minimal number of times the game requires each type of move. This is not an easy problem, even in the case of the Tribonacci numbers, and is left to future work.
\end{rek}


\subsection{Conjectures on Generalized Zeckendorf Games}

\begin{conj}
Player 2 has the winning strategy in the Tribonacci Game for $n$ sufficiently large.
\end{conj}

This is supported by simulation data taken by code in Appendix \ref{sec:Jcode}. Note there are extra difficulties in trying to prove this than occurred in the Fibonacci case. Recall in the proof of Theorem \ref{thm:playertwowins}, we used the fact that certain game states would be found on different layers with opposite parity. Trying to find similar switched parity nodes may be impossible; it seems like all congruent nodes on different layers still occur on turns with the same parity.

\section{Future Work}

There are many more ways that studies of this game can be extended. This paper covered the Generalized Zeckendorf game quite extensively, but improved upper bounds may still be found on the number of moves in any game. This work also showed the existence of a winning strategy for Player two for all $n>2$ in the Fibonacci case (and Player 1 if Player 2 is careless), but it does not show what either of these strategies are.

\begin{itemize}

\item The most natural question is to find a constructive proof that Player 2 has a winning strategy for the Fibonacci game (in other words, what is the winning strategy).\\ \

\item Finding lower bounds on the number of moves and examining who has the winning strategy and how to achieve it for Generalized Zeckendorf games is another natural question. Related to this, we can look at the distribution of the number of moves if the two players randomly move. Numerical investigations in \cite{BEFMcant} suggest that this quantity converges to a Gaussian distribution. Note Gaussian behavior has been seen in related problems in the distribution of the number of summands (see for example \cite{Bes, DG, Ko, MW1, MW2}).\\ \

\item Expanding in another direction, what if more players want to join? Who wins in that case, for either the generalized or regular Zeckendorf game? The analysis done here only shows there is a winning strategy that takes an even number of moves for all $n>2$ for the Fibonacci Zeckendorf game. It says nothing about the number of moves modulo $k$, where $k$ is odd and greater than 2.\\ \

\item What if the game had variable starting points: instead of all ones at the start, a random set of terms in the sequence. How long would the game take then, and does anyone have winning strategies more often?\\ \

\item Finally, can the analysis be performed for more general recurrences than the one in this paper?

\end{itemize}

\appendix


\section{Mathematica Code}\label{sec:Mcode}

Throughout this paper, we use results from code written in Mathematica, available at \begin{center}\bburl{github.com/paulbsmith1996/ZeckendorfGame/blob/master/ZeckGameMathematica.nb}. \end{center}The program contains code for simulating a random version of the Zeckendorf game, running a deterministic worst game algorithm of the Zeckendorf game, and simulating a random Tribonacci Zeckendorf game. There is also code included to tally up the number of moves in each of these simulations, which can be inputting into a graphing function.


\section{Java Code}\label{sec:Jcode}

The following is the ReadMe for the Java applet ``TreeDrawer'' by Paul Baird-Smith found at \bburl{https://github.com/paulbsmith1996/ZeckendorfGame}.

TreeDrawer is used to give a visual representation of the tree structure of the
Zeckendorf game. It plays through a specified game,
determining all moves that can be made, and draw all possible paths to the end of
this game.

Each horizontal layer is composed of GameStates that can be reached in the same
number of moves, namely the depth of the layer (e.g. any state in the 3rd layer
is reached in exactly 3 moves). States with red trim are states at which player 2
has a winning strategy over player 1, and states with blue trim are those at which
player 1 has a winning strategy. Lines between states signify that the lower state
can be reached after a single move from the upper state (parent/child relationship
in the tree structure).

States highlighted in yellow are terminal. There can be at most 1 of these in any
layer by design. Experiments to this point have shown that player 2 always has a
winning strategy (true up to 50), therefore we highlight states in green if they
belong to "the" winning path for player 2 (in reality, there are several winning
paths but we highlight just a single one).

The TreeDrawer can be executed, after compilation, by running the command
\begin{center}
appletviewer TreeDrawer.java
\end{center}
Do not delete the comment in the preamble, as this is used at runtime by the
appletviewer. Email paul.bairdsmith@gmail.com for more information.


\ \\

\end{document}